\documentclass[11pt]{amsart}
\usepackage[utf8]{inputenc}
\usepackage{amsmath}
\usepackage{amsfonts,amssymb}
\usepackage{latexsym}
\usepackage{amsthm}
\usepackage{amsbsy}

\usepackage{tikz-cd}
\usepackage{indentfirst}

\theoremstyle{definition}

\usepackage{fancyhdr} 
\renewcommand{\sectionmark}[1]{\markright{\textit{\thesection. #1}}{} }
\numberwithin{equation}{section}

\newtheorem{theorem}{Theorem}[section]
\newtheorem{proposition}[theorem]{Proposition}
\newtheorem{lemma}[theorem]{Lemma}

\newtheorem{definition}[theorem]{Definition}

\newtheorem{remark}[theorem]{Remark}
\newtheorem{example}[theorem]{Example}
\newtheorem{question}[theorem]{Question}
\newtheorem{open problem}[theorem]{Open problem}

\newcommand{\R}{\mathbb R}
\renewcommand{\Pr}{\mathrm{Pr}}

\newcommand{\Z}{\mathbb Z}
\newcommand{\E}{\mathbb E}
\newcommand{\G}{\mathbb G}

\newcommand{\N}{\mathbb N}
\newcommand{\Q}{\mathbb Q}

\newcommand{\Dom}{\mathrm{Dom}}

\newcommand{\Erg}{\mathrm{Erg}}

\newcommand{\calF}{{\mathcal F}}
\newcommand{\calP}{{\mathcal P}}

\newcommand{\calA}{{\mathcal A}}
\newcommand{\calB}{{\mathcal B}}

\newcommand{\calC}{{\mathcal C}}

\newcommand{\calE}{{\mathcal E}}

\title{Decomposition of the Kantorovich problem and Wasserstein distances on simplexes}
\author{Danila Zaev}
\address{Faculty of Mathematics, Higher School of Economics, Moscow}
\date{\today}
\email{zaev.da@gmail.com}

\begin{document}

\begin{abstract}
Let $X$ be a Polish space, $\mathcal{P}(X)$  be the set of Borel
probability measures on $X$, and $T\colon X\to X$ be a homeomorphism.
We prove that for the simplex $\mathrm{Dom} \subseteq \mathcal{P}(X)$
of all $T$-invariant measures, the Kantorovich metric on
$\mathrm{Dom}$ can be reconstructed from its values on the set of
extreme points. This fact is closely related to the following result:
the invariant optimal transportation plan is a mixture of invariant
optimal transportation plans between extreme points of the simplex. The
latter result can be generalized to the case of the Kantorovich
problem with additional linear constraints and the class of ergodic
decomposable simplexes.
	
\end{abstract}

\maketitle
\tableofcontents
\section{Introduction}
\sectionmark{Introduction}

In this paper we study an important modification of the Kantorovich problem: the problem of mass transportation with additional linear constraints.
The recent developments about general theory of optimal transport are described in \cite{AmbGigli1}, 
\cite{BogKol}, \cite{Vershik1},   \cite{Villani1}.

Transport problem with linear constraints appears naturally in several applications, in particular, in financial mathematics (martingale problem) and in infinite-dimensional analysis (invariant transport problem w.r.t. action of some group), see
\cite{Beigl}, \cite{KolZaev}, \cite{Moameni}, \cite{Zaev}.

Many fundamental problems of mass transport theory (e.g., the problem of existence of a transport map) for measures on infinite-dimensional spaces lead us to study transportation of ergodic measures and the relations between the structure of optimal transport plans and ergodic decompositions (see \cite{KolZaev} for details).

In this paper we provide an answer for the question: when is it possible to ergodically decompose an optimal transport plan into optimal plans between decompositions of marginals?

Let us start with an example. Let $\G$ be a nice enough group (locally-compact, amenable), acting continuously on a compact metric space $(X,d)$. Its action on $X\times X$ is defined in the diagonal way: $g(x_1,x_2):=(g(x_1),g(x_2))$, where $g$ is the action of an element $g\in \G$ on $X$. Measure $\mu$ on $(X,d)$ is called \textbf{invariant} w.r.t. $\G$ iff $\mu\circ g^{-1}=\mu$ for every $g\in \G$. Denote by $\calP_{\G}(X)\subseteq\calP(X)$ the set of all invariant Borel probability measures on $(X,d)$, equipped with the topology of weak convergence and the corresponding Borel $\sigma$-algebra. It is well-known, that $\calP_\G(X)$ is a compact convex non-empty set. As any other convex set, it contains a subset of all extreme points, i.e. such points, that are midpoints of no line segment with endpoints in $\calP_\G(X)$. Let us denote this set via $\partial_e(\calP_\G(X))$ and call it the boundary of $\calP_\G(X)$. It is clear, that extreme points of $\calP_\G(X)$ are exactly $\G$-ergodic measures. It is also known (see \cite{Phelps} for a proof), that $\calP_\G(X)$ is a simplex. It means that each point of $\calP_\G(X)$ can be represented in the unique way as a barycenter of a probability Borel measure concentrated on the boundary $\partial_e(\calP_\G(X))$.

Assume that $\mu$, $\nu$ are some given measures from $\calP_\G(X)$. Consider the following optimization problem, which is usually called Kantorovich problem (\cite{AmbGigli1}, \cite{BogKol}, \cite{Vershik1}, \cite{Villani1}):
$$
\inf\left\{\int c d\pi: \pi\in \calP(X\times X),\ \Pr_1(\pi)=\mu, \Pr_2(\pi)=\nu\right\}.
$$
Here $c:X\times X\rightarrow \R$ is some continuous bounded below function, $\Pr_k(\pi)$ is the $k$-th marginal of measure $\pi$. One can think about this problem as about a problem of finding an optimal way to transfer one mass distribution (measure $\mu$) to the other one (measure $\nu$) with respect to the transfer cost $c$. Measures $\pi$ are called transport plans, and can be thought of as generalized maps from one measure space to the other.

It is quite natural to consider in this situation not the set of all transport plans, but the set of invariant ones w.r.t. the action of $\G$ on $X\times Y$. Then the corresponding the Kantorovich problem has the form:
$$
\inf\left\{\int c d\pi\colon\
\pi\in {\mathcal P}_\G(X\times X), \mathrm{Pr}_1(\pi)=\mu, \mathrm{Pr}_2(\pi)=\nu\right\}
$$
This is an example of Kantorovich problem with additional linear restrictions.
This modification of Kantorovich problem was independently formulated and studied in \cite{Beigl} and \cite{Zaev}.
In our case an additional restriction is the restriction of invariance.

Since $\mu$ and $\nu$ are elements of simplex $\calP_\G(X)$, they can be uniquely represented as follows
$$
\mu=\int_{\partial_e(\mathcal P_\G(X))} \xi_\alpha d\tilde{\mu}(\alpha)=:\mathrm{bar}(\tilde{\mu}),\
\nu=\int_{\partial_e(\mathcal P_\G(X))} \xi_\beta d\tilde{\nu}(\beta)=: \mathrm{bar}(\tilde{\nu}),
$$
i.e. as barycenters of some probability measures concentrated on the boundary of $\calP_\G(X)$: $\tilde{\mu}, \tilde{\nu}\in \calP(\partial_e(\calP_\G(X)))$. It is natural to ask, whether it is possible to decompose Kantorovich problem in the following way:
\begin{multline}
\label{example of theorem 1}
\inf\left\{\int c d\pi\colon\, \pi\in {\mathcal P}_\G\G(X\times X),\ \mathrm{Pr}_1(\pi)=\mu, \mathrm{Pr}_2(\pi)=\nu\right\}=\\
=\inf\left\{\int \inf\left\{\int c d\pi\colon\, \pi\in {\mathcal P}_\G(X\times X),\ \mathrm{Pr}_1(\pi)=\xi_\alpha, \mathrm{Pr}_2(\pi)=\xi_\beta\right\}
d\tilde{\pi}\colon\, \tilde{\pi}\in \tilde{\Pi}(\tilde{\mu},\tilde{\nu})\right\},
\end{multline}
where $\tilde{\Pi}(\tilde{\mu},\tilde{\nu}):=\left\{\pi\in \calP(\partial_e(\calP_\G(X))\times \partial_e(\calP_\G(X))): \Pr_1(\pi)=\tilde{\mu},\ \Pr_2(\pi)=\tilde{\nu}\right\}$. In other words, is it possible to split Kantorovich problem into two sub-problems: the first one is to calculate the minimal cost of transportation between any pair of measures \textbf{from} the boundary of a simplex, and the second one is to find an optimal way to transport measures \textbf{on} the boundary, i.e. measures on the set of ergodic measures? In this paper we describe sufficient conditions for existence of such a decomposition. It appears, that this decomposition result holds for a special kind of simplexes $\Dom\subseteq \calP(X)$, known as ergodic decomposable simplexes (\cite{Dynkin}, \cite{Kerstan}). The simplex of invariant measures, described above, is an example of an ergodic decomposable simplex. In particular, the following statement is valid.

\begin{theorem}
	\label{1 main theorem}
	Let $(X, {\mathcal A})$, $(Y, {\mathcal B})$ be two polish spaces with Borel $\sigma$-algebras with given continuous actions of the group $(\Z, +)$, ${\mathcal A}^{inv}\subseteq {\mathcal A}$, ${\mathcal B}^{inv}\subseteq {\mathcal B}$ be corresponding $\sigma$-subalgebras of invariant sets,
	$c\colon\, X\times Y \to  \R$ be some lower semicontinuous and bounded below function.
	Then the equality (\ref{example of theorem 1}) is valid,
	where $\tilde{\mu}$ is the restriction of $\mu$ on ${\mathcal A}^{inv}$,  and $\tilde{\nu}$ is the restriction
	of $\nu$ on ${\mathcal B}^{inv}$.
\end{theorem}

This theorem is a particular case of the main theorem of the paper (Theorem \ref{main theorem}). For the proof let us consider ergodic decomposition of an optimal invariant transport plan. It is not hard to check, that this decomposition consists of ergodic transport plans with ergodic marginals. It is required to show that almost all plans in this decompositions are optimal. If it is not true, one can replace non-optimal components with the optimal ones and obtain a contradiction. Meanwhile, there is a technical difficulty here: making such a replacement one need to be careful with measurability. The family of transport plans consisting the decomposition should remain to be measurably dependent on marginals. This issue is discussed in details in Section 4 of the paper.

Let $\Dom\subseteq \calP(X)$ be a simplex that admits decomposition of a Kantorovich problem. Denote for shortness $E:=\partial_e(\Dom)\subset \Dom$.
Assume that a distance function $\bar{d}: E\times E \rightarrow \R_{\geq 0}$ is defined on $E$, and that $\bar d$ metricizes the topology of $E$.
If we fix some $p\in [1,\infty)$, we can extend $\bar{d}$ to a distance function $\hat{d_p}$ on $\Dom$ via the formula:
\begin{equation}
\label{extension of distance}
\hat{d_p}(\mu,\nu):=\inf\left\{\left(\int \bar{d}^p(e_1, e_2)d\pi\right)^\frac{1}{p}\colon\ \pi\in {\mathcal P}(E\times E), \mathrm{Pr}_1(\pi)=\tilde{\mu}, \mathrm{Pr}_2(\pi)=\tilde{\nu}\right\},
\end{equation}
where ${\mathcal P}(E\times E)$ is the set of all Borel probability measures on $E\times E$, $\mathrm{Pr}_k$ is an operator sending a measure from $\calP(E\times E)$ to its $k$-th marginal, $\mathrm{bar}(\tilde{\mu})=\mu$, $\mathrm{bar}(\tilde{\nu})=\nu$.

Let us recall the definition of the $L^p$-Wasserstein distance in the case the measures are defined on a compact metric space.
\begin{definition}
	Let $(X,d)$ be a compact metric space, $\calP(X)$ be a set of all Borel (= Radon) probability measures on it. Then for any $p\in [1,\infty)$ $L^p$-Wasserstein distance $W_p:\calP(X)\times \calP(X) \to \R_{\geq 0}$ is a distance function defined by the formula:
	\begin{equation}
	\label{wasserstain distance}
	W_p(\mu,\nu):=\inf\left\{\left(\int d^p(x, y)d\pi\right)^\frac{1}{p}\colon\ \pi\in {\mathcal P}(X\times X),\ \mathrm{Pr}_1(\pi)=\mu,\ \mathrm{Pr}_2(\pi)=\nu\right\}.
	\end{equation}
\end{definition}

Considering invariant Kantorovich problem, it natural to introduce the following analogue $L^p$-Wasserstein distance on $\mathrm{Dom}:={\mathcal P}_\G(X)$:
\begin{equation}
\label{invariant wasserstain distance}
W_p^\G(\mu,\nu):=\inf\left\{\left(\int d^p(x, y)d\pi\right)^\frac{1}{p}\colon\
\pi\in {\mathcal P}_\G(X\times X), \mathrm{Pr}_1(\pi)=\mu, \mathrm{Pr}_2(\pi)=\nu\right\}.
\end{equation}
It is known (but it is not a trivial fact), that in the case $d$ is invariant w.r.t. $\G$, function $W_p^\G$ coincides with $W_p$ on $\Dom$ (see Moameni, \cite{Moameni}). We do not assume invariance of a distance $d$, thus $W_p^\G\geq W_p$ on their common domain of definition.
It can be checked, that $W_p^\G$ actually satisfies every axiom of a distance function on $\calP_\G(X)$ (see Example \ref{example invariant problem 2} for the proof).

\begin{question}
	Let for some $p\in[1,+\infty)$ the distance $W_p^\G$ is defined on $\mathrm{Dom}$. Is it possible to define a distance
	$\bar{d}$ on $E = \partial_e \mathrm{Dom}$ in such a way, that its
	extension $\hat d_p$ on $\mathrm{Dom}$ would coincide with $W_p^\G$?	
\end{question}
Let $\bar{d}_p^\G$ be the restriction of the invariant $L^p$-Wasserstein distance (defined via formula (\ref{invariant wasserstain distance})) to $\Erg(X):=\partial_e(\mathrm{Dom})$, i.e.

\begin{equation}
\label{invariant wasserstein distance on boundary}
(\bar{d}_p^\G)^p(\xi_\alpha, \xi_\beta):=\inf\left\{\int d^p(\xi_\alpha, \xi_\beta)d\pi
\colon\ \pi\in {\mathcal P}_\G(X\times X), \mathrm{Pr}_1(\pi)=\xi_\alpha, \mathrm{Pr}_2(\pi)=\xi_\beta\right\}.\end{equation}
Then it is true, that
$$
W_p^\G(\mu,\nu)=\hat{d_p^\G}(\mu,\nu)
$$
for every $\mu,\nu\in \mathrm{Dom}$. Here $\hat{d_p^\G}$ is the extension of $\bar{d_p^\G}$ (\ref{invariant wasserstein distance on boundary})
from $E:=\partial_e(\mathrm{Dom})$ on $\mathrm{Dom}$ defined by the formula (\ref{extension of distance}).

This fact is a particular case of Theorem \ref{theorem restricted wasserstein distance on a simplex}. It is clear, that it is closely related to the decomposition of the Kantorovich problem.
Denote via $({\mathcal P}_\G(X), W_p^\G(d))$ the space of all invariant probability measures on a metric compact space $(X,d)$, and equip it with the invariant $L^p$-Wasserstein distance.
Since
$$
\partial_e(\mathrm{Dom})\subseteq \mathrm{Dom} = {\mathcal P}_\G(X),
$$
one can restrict $W_p^\G(d)$ on the subsets $\mathrm{Dom}$ and $\partial_e(\mathrm{Dom})$
and obtain metric spaces $(\mathrm{Dom}, W_p^\G(d))$ and $(\partial_e(\mathrm{Dom}), W_p^\G(d))$
respectively.
By the definition of simplex,
$$
\mathrm{Dom} \simeq {\mathcal P}(\partial_e(\mathrm{Dom})),
$$
where ``$\simeq$'' is for homeomorphism. Hence one can construct $L^p$-Wasserstein distance
on ${\mathcal P}(\partial_e(\mathrm{Dom}))$ considering
$(\partial_e(\mathrm{Dom}), W_p^\G(d))$ as the original metric structure.
Denote the constructed metric space via
$({\mathcal P}(\partial_e(\mathrm{Dom})), W_p(W_p^\G(d)))$. Due to the decomposition result,
$$
({\mathcal P}(\partial_e(\mathrm{Dom})), W_p(W_p^\G(d)))=(\mathrm{Dom}, W_p^\G(d)),
$$
where ``$=$'' is for isometric isomorphism.
We shall call results of this type as decompositions of a Wasserstein distance.

We are going to prove a general result about decompositions of transport plans, such that the described decompositions of Wasserstein distances appear to be its particular cases. The main theorems of this paper are Theorem \ref{main theorem} and Theorem \ref{theorem restricted wasserstein distance on a simplex}. The first one establishes a decomposition result for a Kantorovich problem, while the second one describes decomposition for Wasserstein-like distances on simplexes. The preliminary section of the paper contains all required definitions and facts from the theory of simplexes, sufficient statistics, and ergodic decompositions of measures. Describing this theory, we mostly follow \cite{Dynkin} and \cite{Kerstan}.

In Section 3 we define the Kantorovich problem with additional restrictions.
We also formulate notions of ``good'' (in some defined sense) linear restrictions: \textit{weakly regular}, \textit{ergodic decomposable}, \textit{coherent} and \textit{geometric} ones. These notions are important for the formulation of the main results.

In Section 4 we provide a proof of a measurable selection statement, based on the results of Rieder \cite{Rieder} about measurability of solutions in optimization problems.

In Section 5 we formulate and prove the main statement about existence of the decomposition of a Kantorovich problem in the case an additional restriction is ergodic decomposable, weakly regular, and coherent. In Section 6 we prove the result about decomposition of Wasserstein-like distances in the case an additional restriction is geometric (in addition to the previous assumptions).

We also describe some possible applications of the obtained results. These applications are closely related to the study of symmetric and invariant modifications of Kantorovich problem, which are studied, for example, in \cite{Beigl}, \cite{ColMar}, \cite{KolZaev}, \cite{Zaev}.
 	
\section{Preliminaries: simplexes and disintegration of measures}
\sectionmark{Preliminaries: simplexes and disintegration of measures}
In this section we briefly discuss definitions and results from the theory of simplexes, ergodic decompositions and disintegrations of measures, that will be used in the rest of the paper.

The standard way to define a notion of infinite-dimensional simplex is the approach of Choquet.

\begin{definition}
	A point $x\in K$ of a convex compact set $K$ is extreme iff for any two points $a, b\in K$, such that $x=t\cdot a + (1-t)\cdot b$ for some
	$t\in (0,1)$, it follows that $a=b$.
	Compact convex metrizable set $K$ is a \textbf{Choquet simplex} iff for every element $a\in K$ of this set there is a \textbf{unique} Borel measure $\mu_a$ on $K$ such that
	\begin{itemize}
		\item $\mu_a(\partial_e(K))=1$, where $\partial_e(K)$ is a subset consisted of all extreme points of $K$,
		\item $a=\mathrm{bar}(\mu_a):=\int_K x d \mu_a$.
	\end{itemize}
\end{definition}
\textbf{Affine map} $T: K_1\rightarrow K_2$ between two convex sets is a map with the property $T(\alpha\cdot a + (1-\alpha)\cdot b)=\alpha\cdot T(a)+(1-\alpha)\cdot T(b)$ for all $\alpha\in[0,1]$, $a,b\in K_1$. Two simplexes are isomorphic iff there exists an affine homeomorphism between them. There is only one $n$-dimensional Choquet simplex up to isomorphism for any finite $n$.

\textbf{Bauer simplex} is a Choquet simplex with closed (hence compact) Choquet boundary. There are infinitely many non-isomorphic Bauer simplexes. \textbf{Poulsen simplex} is a Choquet simplex with dense Choquet boundary. There is only one (up to isomorphism) Poulsen simplex (see \cite{Phelps} for details).

\begin{example}
	Let $(X,d)$ be a compact metric space, $\calP(X)$ be the set of all probability measures equipped with the weak$^\ast$-topology, induced by the embedding $\calP(X)\subset C(X)^\ast$. Then $\calP(X)$ is a Bauer simplex, its Choquet boundary consists of all Dirac measures. Every Bauer simplex is isomorphic to $\calP(X)$ for some compact metrizable space $X$ (see \cite{Batty}).
\end{example}

\begin{example}
	Let $(X,d)$ be a compact metric space, $G$ be an amenable group, acting on $X$ by homeomorphisms, $\calP_G(X)$ be the set of all invariant probability measures equipped with the weak$^*$-topology, induced by the embedding $\calP(X)\subset C(X)^\ast$. Then $\calP_G(X)$ is a Choquet simplex, its Choquet boundary consists of all ergodic measures. Every Choquet simplex is isomorphic to $\calP_G(X)$ for some compact metrizable $X$ and some continuous action of the group $\Z$ (see \cite{Cortez}, \cite{Downarowicz}).
\end{example}

\begin{example}
	Let $\calC$ be a $C^\ast$-algebra, $S(\calC)$ be its state space, i.e. set of all positive continuous linear functionals of norm one, $S(\calC)$ be equipped with weak$^\ast$-topology, induced by the embedding $S(\calC)\subset \calC^\ast$. Then $S(\calC)$ is a compact convex set. It is a Choquet simplex if and only if $\calC$ is a commutative algebra (see \cite{Batty}).
\end{example}

\begin{example}
	Let $(X,d)$ be a Polish metric space, $\calP(X)$ be the set of all probability measures equipped with the weak$^\ast$-topology, induced by the embedding $\calP(X)\subset C(X)^\ast$. Then $\calP(X)$ is a Choquet (and Bauer) simplex if and only if $X$ is compact.
\end{example}

It can be noted, that this approach to the notion of simplex has two substantial disadvantages:
\begin{itemize}
	\item At first, all Choquet simplexes are assumed to be compact sets by the definition. It is convenient in many cases, but spaces of measures, being equipped with Kantorovich metric, are not compact in general.
	
	\item The second, and, arguably, more important ``contra'' for the Choquet theory:  it does not link explicitly the representation of an element of simplex as a mixture of the extreme points with an ergodic decomposition of the measure corresponding to the element.
\end{itemize}

The approach, introduced by Dynkin in \cite{Dynkin} and developed subsequently by several authors in \cite{Kerstan} is free of these disadvantages. The notion of simplex provided by Dynkin is a generalization of the one introduced by Choquet. It is formulated in purely measure-theoretic terms and does not require any topological assumptions. For a special subclass of Dynkin simplexes (that are called \textit{ergodic decomposable} in \cite{Kerstan}) there is a result that connects representation of a measure as a mixture of extremes with its disintegration w.r.t. some appropriate $\sigma$-subalgebra. Conditional measures of this disintegration correspond to the extreme points of a simplex.

Let us denote via $B(X,\calA)$ the set of all $\calA$-measurable bounded real-valued functions on $X$.
\begin{definition}
	Let $(X,\calA)$ be a measurable space, $M\subset\calP(X)$ is a subset of probability measures. The smallest $\sigma$-algebra on $M$, such that for any $f\in B(X,\calA)$, $\mu\in M$ the map $\mu\rightarrow \mu(f)$ is measurable, is called a \textbf{standard} $\sigma$-algebra.
\end{definition}

\begin{definition}
	\label{measurable convex set of probability measures}
	Let $(X,{\mathcal A})$ be a measurable space, $M\subseteq{\mathcal P}(X)$ be a subset equipped with a standard $\sigma$-algebra.
	Let us define \textbf{barycenter} $\mathrm{bar}\colon\ {\mathcal P}(M)\to M$ by the formula
	$$
	\mathrm{bar}(\tilde{\mu})(f)=\int_M \left(\int_X f d\nu \right)d\tilde{\mu}(\nu),\ \forall f\in B(X,{\mathcal A}).
	$$
	The \textbf{boundary} of $(M, {\mathcal B})$ is a set $\partial_e(M)\subset M$ of such points $m$,
	that for any measure $\mu\in {\mathcal P}(M)$ the property $\mathrm{bar}(\mu)=m$ implies $\mu_m(\{m\})=1$.
	A measurable space $(M, {\mathcal B})$ is called Dynkin simplex,
	iff its boundary $\partial_e(M)$ is measurable, and for any $\mu\in M$ there exists a unique $\tilde{\mu}\in {\mathcal P}(M)$ s.t. $\mathrm{bar}(\tilde{\mu})=\mu$
	and $\tilde{\mu}(\partial_e(M))=1$.
\end{definition}

\begin{example}
	\label{example 1 of simplex}
	If $M:={\mathcal P}(X)$, then it is a Dynkin simplex with boundary $\partial_e(M)$ consisted of all Dirac measures.
\end{example}

\begin{example}
	\label{example 2 of simplex}
	Let $(X,\calA)$ be a measurable space, $\G$ be an amenable group, acting on $X$ by measurable transformations, $\calP_\G(X)$ is a set of all invariant measures. 
	If $M=\calP_\G(X)$,  $\calB$ is a $\sigma$-algebra generated by evaluations, then $(M,\calB)$ is a Dynkin simplex, its Dynkin boundary consists of all ergodic measures. 
	Since any Choquet simplex is isomorphic to $\calP_\G(X)$ for some compact metrizable $X$ and some continuous action of the group $\Z$ (see. \cite{Cortez}, \cite{Downarowicz}), every Choquet simplex is a Dynkin simplex, up to an isomorphism. 
\end{example}
Let $(X,\calA)$ be some measurable space.
\begin{definition}
	A function $Q:X\times \calA \rightarrow [0,1]$ is called \textbf{Markov kernel} on $(X,\calA)$ iff for each $x\in X$, $Q^x:=Q(x,\cdot)$ is a probability measure on $\calA$, and for each $A\in \calA$, $x\rightarrow Q(x,A)$ is a $\calA$-measurable function.
\end{definition}
We call two Markov kernels $Q_1$, $Q_2$ on $(X,\calA)$ $M$-equivalent for some $M\subseteq \calP(X)$ iff $\mu(\{x: Q_1(x,\cdot)=Q_2(x,\cdot) \})=1$ for all $\mu\in M$.

\begin{definition}
	Let $(X,\calA)$, $(Y, \calB)$ be two measurable spaces. A function $Q:Y\times \calA \rightarrow [0,1]$ is called \textbf{Markov transition kernel} from $(X,\calA)$ to $(Y,\calB)$ iff for each $y\in Y$, $Q^y:=Q(y,\cdot)$ is a probability measure on $(X,\calA)$, and for each $A\in \calA$, $y\rightarrow Q(y,A)$ is a $\calB$-measurable function.
\end{definition}
We call two Markov transition kernels $Q_1$, $Q_2$ from $(X,\calA)$ to $(Y,\calB)$ $M$-equivalent for some $M\subseteq \calP(Y)$ iff $\mu(\{y: Q_1(y,\cdot)=Q_2(y,\cdot) \})=1$ for all $\mu\in M$.
Each Markov kernel $Q$ defines a positive operator on $B(X,\calA)$ by the formula
$$
Q(f)(x):=\int f dQ(x,\cdot).
$$
\begin{definition}
	Let $M\subseteq \calP(X,\calA)$, $\calA^0$ be a $\sigma$-subalgebra of $\calA$.
	A Markov kernel $Q$ on $(X,\calA)$ is called a \textbf{decomposition} for a triple $(\calA,\calA^0,M)$ iff 
	$$
	\E_\mu(f|\calA^0)=Q(f)\ \text{a.e. x w.r.t. $\mu$}
	$$
	for all $f\in B(X,\calA)$, $\mu \in M$. Here $\E_\mu(f|\calA^0)$ is a conditional expectation of $f$ with respect to $\sigma$-algebra $\calA^0$ and measure $\mu$.
\end{definition}
Note, that a decomposition for a triple $(\calA,\calA^0,M)$ is by definition a Markov transition kernel from $(X,\calA)$ to $(X,\calA^0)$.

Define the action of Markov kernel on $\calP(X,\calA)$ as follows:
$$
Q_\#(\mu)(A):=\int_X Q(x,A) d\mu\ \forall A\in \calA.
$$
A measure $\mu\in \calP$ is said to be invariant w.r.t. $Q$ iff $Q_\#(\mu)=\mu$.

Analogously, for any Markov transition kernel there is a map from $B(X,\calA)$ to $B(Y,\calB)$:
$$
Q(f)(y):=\int f(x) dQ(y,\cdot)
$$
and a map from $\calP(Y,\calB)$ to $\calP(X,\calA)$
$$
Q_\#(\nu)(A):=\int_Y Q(y,A) d\nu\ \forall A\in \calA.
$$

\begin{definition}
	$M\subseteq \calP(X)$ is \textbf{separable} iff $\calA$ contains a countable family $\calF$ of subsets, such that for each pair of different measures $\mu_1, \mu_2 \in M$ there exists $A\in \calF$ s.t. $\mu_1(A)\neq \mu_2(A)$. 
\end{definition}

\begin{definition}
	A $\sigma$-algebra $\calA^0\subseteq \calA$ is called \textbf{sufficient} for a separable $M\subseteq \calP(X)$, if there is a Markov kernel $Q$ on $(X,\calA)$, such that it is a decomposition for a triple $(\calA,\calA^0,M)$.
\end{definition}

\begin{definition}
	A sufficient $\sigma$-algebra is called \textbf{$H$-sufficient} for $M$ if it is sufficient and 
	$$
	\mu(\{x: Q^x\in M\})=1\ \forall \mu\in M.
	$$
\end{definition}
We call two $\sigma$-subalgebras $\calA^1\subseteq \calA$, $\calA^2\subseteq \calA$ $M$-equivalent for some $M\subseteq\calP(X,\calA)$ iff for any $\mu\in M$, any $A_1\in \calA^1$ there exists $A_2\in \calA^2$ s.t. $A_1=A_2$ a.e. w.r.t. $\mu$.

The following important theorem is the union of the statements of Theorem 3.1, Theorem 3.2, Theorem 3.3 of \cite{Dynkin} and Lemma 3.6 of \cite{Kerstan}.
\begin{theorem}
	\label {equivalent definitions of ergodic decomposable simplex}
	Assume that $\calA$ is a countably generated $\sigma$-algebra. Then for a separable set $M\subseteq \calP(X,\calA)$, the following properties are equivalent.
	\begin{itemize}
		\item There exists a unique, up to $M$-equivalence, $H$-sufficient $\sigma$-algebra $\calA^0\subseteq \calA$ for $M$, which is the $\sigma$-algebra of all sets $A\in \calA$ s.t. $\mu(A)=0$ or $\mu(A)=1\ $ $\forall \mu \in M$.
		\item There exists a Markov kernel $Q$ on $(X,\calA)$ with the property:
		$$
		Q(g Q(f))=Q(g)Q(f)\ \forall f,g\in B(X,\calA),
		$$
		or, equivalently, with the property
		$$
		Q^x(\{y\in X: Q^x=Q^y \})=1\ \forall x\in X,
		$$	
		such that $M$ is a set of all $Q$-invariant elements of $\calP(X,\calA)$.
	\end{itemize}
	For a Markov kernel from the proposition, the corresponding $H$-sufficient algebra is the algebra (up to $M$-equivalence) of all $Q$-invariant measurable sets.
\end{theorem}

\begin{definition}
	A separable set $M\subseteq \calP(X,\calA)$ of measures on a countably generated $\sigma$-algebra $\calA$ is called \textbf{ergodic decomposable simplex} iff any of the equivalent properties of Theorem \ref{equivalent definitions of ergodic decomposable simplex} is satisfied. 
\end{definition}

In \cite{Dynkin} Dynkin provide a series of examples of ergodic decomposable simplexes. In particular, the simplexes from examples \ref{example 1 of simplex}, \ref{example 2 of simplex} are ergodic decomposable.

The following result is a direct consequence of Theorem 3.1 from \cite{Dynkin} and Remark 3.8 from \cite{Kerstan}.
\begin{theorem}
	Ergodically decomposable simplex $M\subseteq \calP(X,\calA)$ is a Dynkin simplex. Moreover, there exists such a Markov kernel $Q$ from the definition of ergodic decomposable simplex, that Dynkin boundary is defined as following:
	$$
	\partial_e(M)=\{Q^x: x\in X\}
	$$
	and 
	$$
	\mathrm{bar}(\tilde{\mu})=\mu \iff \tilde{\mu}(S)=\mu(\{x: Q^x\in S\})
	$$
	for any measurable $S\subseteq M$.
\end{theorem}

The following statement establishes a coincidence (in the case of a Polish space $X$) of the Borel $\sigma$-algebra generated by the topology of weak convergence on $\calP(X)$, and the standard $\sigma$-algebra on $\calP(X)$.
\begin{proposition}
	\label{relation between weak and evaluation algebras}
	Let $(X,\calA)$ be a Polish space with Borel $\sigma$-algebra. Then the following three $\sigma$-algebras on $\calP(X)$ coincide:
	\begin{enumerate}
		\item the standard $\sigma$-algebra $\calE$,
		\item $\sigma$-algebra, generated by all sets of the form $\{\mu\in \calP(X): a\leq \mu(A)\leq b\}$, $A\in \calA$, $a, b\in\Q\cap[0,1]$,
		\item Borel $\sigma$-algebra, generated by the topology of weak convergence on $\calP(X)$.
	\end{enumerate}
\end{proposition}
\begin{proof}
	Let us show first, that $\calE$ can be generated by all functionals of the form $\mu\rightarrow \mu(A)$, $A\in \calA$. It is obvious, that if all maps $\mu \rightarrow \mu(A)$ are measurable, then the map $\mu \rightarrow \mu(s):=\int s d\mu$ is measurable for every step-function $s$ (step function is a linear combination of a finite number of measurable indicator functions). By definition of Lebesgue integration, $\mu(f):=\int f d\mu:=\sup\{\mu(s): 0\leq s\leq f^{+}\}-\sup\{\mu(s): 0\leq s\leq f^{-}\}$ (where $f=f^{+}-f^{-}$, $f^{+}, f^{-}$ are non-negative and measurable, $s$ is a measurable step-function) and its measurability follows from the classic fact, that a pointwise supremum of measurable maps is measurable. 
	
	Since the family of all intervals with rational endpoints in $[0,1]$ generates Borel $\sigma$-algebra on $[0,1]$, the standard $\sigma$-algebra coincides with the $\sigma$-algebra generated by all sets of the form $\{\mu\in \calP(X): a\leq \mu(A)\leq b\}$, $A\in \calA$, $a, b\in\Q\cap[0,1]$.
	
	The equivalence of this $\sigma$-algebra with the Borel one (w.r.t. topology of weak convergence) was proved in \cite{Gaudard} (Theorem 2.3).
\end{proof}
%
%
%

Let us discuss two important examples of ergodic decomposable simplexes.
\begin{example}[Main example]
	\label{example invariant problem 1}
	Let $\mathbb G$ be locally compact amenable group with a fixed continuous action on a Polish metric space $(X,d)$.
	An action of $\mathbb G$ on $X\times X$ is defined in the ``diagonal''
	way: $g(x_1,x_2):=(g(x_1),g(x_2))$, where $g$ is an action of the element $g\in \mathbb G$ on $X$. A measure $\mu$ on $(X,d)$ is called \textit{invariant}
	w.r.t. $\mathbb G$ iff $\mu\circ g^{-1}=\mu$ for every $g \in \mathbb G$.
	Denote via ${\mathcal P}_{\mathbb G}(X)\subseteq{\mathcal P}(X)$
	the set of all Borel invariant probability measures on $(X,d)$. 
	It is known that ${\mathcal P}_{\mathbb G}(X)$ is a closed ergodic decomposable simplex
	(see Sections 6 and 7 in \cite{Dynkin} for the proof). The corresponding
	$H$-sufficient $\sigma$-algebra can be defined as an algebra ${\mathcal A}^{inv}$
	of all Borel $\mathbb G$-invariant measurable subsets.
	The corresponding Markov kernel $Q$ on $(X,{\mathcal A})$ is defined as
	$$
	\tilde{Q}(x,A):=\lim_{n\to  \infty} \frac{1}{\mu(F_n)} \int_{F_n} {g}_{\#} (\delta_x(A)) d\nu(g).
	$$
	Here $(F_n)$ is a F\"olner sequence for $\mathbb G$, $\nu$ is a left Haar measure on $\mathbb G$, $\delta_x$ is the Dirac measure concentrated in $x\in X$.
	Recall that a F\"olner sequence for $\mathbb G$ is a sequence of nonempty compact subsets of $\G$, such that
	$$
	\lim_{n\to\infty}\frac{\mu(gF_n\bigtriangleup F_n)}{\mu(F_n)}= 0
	$$
	for each
	$g\in \mathbb G$, where $gF_n:=\{gf\colon\  f\in F_n\}$, and $\bigtriangleup$ is the symmetric difference.
\end{example}

\begin{example}[Discrete-time Markov process]
	\label{example markov process}
	Let $(X,{\mathcal A})$ be a Polish space with Borel $\sigma$-algebra,
	$Q\colon\  X\times {\mathcal A}\to  [0,1]$ be a Markov kernel with the property $Q(f Q(g))=Q(f) Q(g)$. The sequence of probability measures $\mu_n:=Q^n(\mu_0)$, $n\in \N$, is called a discrete-time Markov process with the initial distribution $\mu_0$ and the transition kernel $Q$.
	
	In this case the set of all $Q$-invariant probability measures on $X$ (stationary distributions) $\calP_{Q}(X)\subseteq \calP(X)$ is an ergodic decomposable simplex. The simplex is closed if the corresponding Markov operator $Q: \calP(X)\rightarrow \calP(X)$ is weak$^*$-continuous.
	Its $H$-sufficient $\sigma$-algebra is equivalent to the algebra $\calA_{Q}$ of all Borel $Q$-invariant subsets (i.e. such $A\in \calA$, that $Q(I_A)=I_A$, where $I_A$ is an indicator function of $A$). See Section 9 of \cite{Dynkin} for the generalization of this example to a continuous-time case.
\end{example}

\section{Kantorovich problem with additional linear restrictions}
\sectionmark{Kantorovich problem with additional linear restrictions}
Let $(X,{\mathcal A})$, $(Y,{\mathcal B})$ be two Polish spaces with Borel $\sigma$-algebras.
Denote via ${\mathcal P}(X)$, ${\mathcal P}(Y)$, ${\mathcal P}(X\times Y)$
the sets of all probability measures on $(X,{\mathcal A})$, $(Y,{\mathcal B})$
and $(X\times Y, {\mathcal A}\otimes {\mathcal B})$ respectively.
Note that it is possible to consider sets of measures as subsets of Banach dual spaces: ${\mathcal P}(X)\subset C_b(X)^\ast$,
${\mathcal P}(Y)\subset C_b(Y)^\ast$, ${\mathcal P}(X\times Y)\subset C_b(X\times Y)^\ast$.
Consider weak$^*$-topologies on
dual spaces and equip the spaces of probability measures with the topologies induced by the inclusion. 
It is known, that the defined topological spaces ${\mathcal P}(X)$, ${\mathcal P}(Y)$, ${\mathcal P}(X\times Y)$
are Polish (see \cite{Parthasarathy}, Theorem 6.2 and 6.5).
Moreover, their topology coincides with the topology of weak convergence of measures. Due to this fact, the notions of weak and weak$^*$-topology on the spaces of probability measures will be used interchangeably.

Let us define projection operators $\mathrm{Pr_X}\colon\ {\mathcal P}(X\times Y) \to  {\mathcal P}(X)$,
$\mathrm{Pr_Y}\colon\ {\mathcal P}(X\times Y) \to  {\mathcal P}(Y)$ as follows:
\begin{eqnarray}
\mathrm{Pr_X}(\pi)(A)=\pi(A\times Y),\ \forall A\in {\mathcal A},\\
\mathrm{Pr_Y}(\pi)(B)=\pi(X\times B),\ \forall B\in {\mathcal B}.
\end{eqnarray}
In addition, let us define
$\mathrm{Pr}\colon\ {\mathcal P}(X\times Y) \to  {\mathcal P}(X)\times {\mathcal P}(X)$
as $\mathrm{Pr}(\pi)=(\mathrm{Pr_X}(\pi),\mathrm{Pr_Y}(\pi))$.
It is clear, that the defined operators are weakly continuous.

In Kantorovich theory we are interested in the following sets of measures:
$$
\Pi(\mu,\nu)=\{\pi\in {\mathcal P}(X\times Y)\colon\ \mathrm{Pr}(\pi)=(\mu,\nu)\},\ \mu\in \mathcal P(X),\ \nu\in\mathcal P(Y).
$$
Elements of these sets are called \textbf{transport plans}.
Since the map $\mathrm{Pr}$ is continuous,
the set $\Pi(\mu,\nu)$ is closed. It is also known to be compact (see \cite{AmbGigli1}, \cite{BogKol}, \cite{Villani1}).

Let us define \textbf{cost functional} as a functional $C\colon\  {\mathcal P}(X\times Y) \to  \R\cup\{+\infty\}$ that is affine:
$$
\alpha C(\pi) + (1-\alpha) C(\gamma)=C (\alpha\pi+(1-\alpha)\gamma)
$$
for all $\alpha\in [0,1]$, $\pi, \gamma\in {\mathcal P}(X\times Y)$, $0\cdot (+\infty):=0$.
In most cases we are interested in weakly lower semi-continuous (l.s.c.) cost functionals.
Examples of such cost functionals can be provided by integration of lower semi-continuous functions bounded below. If $c: X\times Y \rightarrow \R$ is such a function, then $Cost(\pi):=\int c d\pi$ is a cost functional. Meanwhile, the variety of cost functionals is not reduced to the functionals of this form.
For example, in \cite{Vershik2} was introduced a more general class of  \textit{virtually-continuous} functions. Integration of bounded below function of this class also defines a weakly l.s.c. functional as well.

Let us fix some cost functional $C$ and define two sets of optimal plans:
$$
\Pi^{opt}(X\times Y):=\{\pi\in {\mathcal P}(X\times Y)\colon\  C(\pi)\leq C(\gamma),\ \forall \gamma\in \Pi(\mathrm{Pr}(\pi))\},
$$
$$
\Pi^{opt}(\mu,\nu):=\{\pi\in \Pi(\mu,\nu)\cap \Pi^{opt}(X\times Y)\}.
$$

In some applications the Kantorovich problem appears in a modified form. For example, one may be interested in the search of optimal element not in the set of all transport plans, but only among invariant (in some defined sense) ones.
Such problem is sometimes called \textit{invariant} or \textit{symmetric} Kantorovich problem. Another example is martingale Kantorovich problem, where optimal solution is seek in the set of plans-martingales.
A fruitful way to formalize and generalize such modifications is the notion of the Kantorovich problem with additional linear restriction,
which is the main object of interest in this section.

Many aspects of the Kantorovich problem in the symmetric context were described in \cite{KolZaev}. See also \cite{ColMar} for the results about symmetric Monge-Kantorovich problem in the multi-marginal setting, \cite{Beigl} and \cite{Zaev} for duality and monotonicity results about Kantorovich problem with additional linear restriction. Let us start with a formalization of the notion of additional linear restriction.

\begin{definition}
	For a given pair of measurable spaces $(X,{\mathcal A})$, $(Y,{\mathcal B})$
	let us call by \textbf{linear restriction} a triple $R=(\Omega, M^X, M^Y)$,
	consisted of a subset of measurable functions
	$\Omega\subseteq L^0(X\times Y, {\mathcal A}\otimes {\mathcal B})$
	and two nonempty sets of measures $M^X\subseteq {\mathcal P}(X)$, $M^Y\subseteq {\mathcal P}(Y)$.
\end{definition}

Let us define bounded sets of transport plans as follows:
\begin{equation}
\Pi_R(X\times Y):=\left\{\pi\in {\mathcal P}(X\times Y)\colon\  \int \omega d\pi=0\ \forall \omega\in \Omega,\ \mathrm{Pr_X}(\pi)\in M^X,\ \mathrm{Pr_Y}(\pi)\in M^Y\right\},
\end{equation}
\begin{equation}
\Pi_R(\mu,\nu):=\Big\{\pi\in \Pi_R(X\times Y)\colon\  \mathrm{Pr}_X(\pi)=\mu,\ \mathrm{Pr}_Y(\pi)=\nu\Big\}.
\end{equation}
Two linear restrictions $R_1$, $R_2$ are called equivalent iff
$\Pi_{R_1}(X\times Y)=\Pi_{R_2}(X\times Y)$.

Everywhere in this section we assume that both
$(X,{\mathcal A})$, $(Y,{\mathcal B})$ are \textit{Polish} spaces
with \textit{Borel} $\sigma$-algebras.

Let us formulate a property of a linear restriction, which will be used in the next section to obtain an existence result for measurable selection.

\begin{definition}
	Let us call linear restriction $R=(\Omega, M^X, M^Y)$ weakly regular iff
	\begin{enumerate}
		\item $M^X$, $M^Y$ are closed in the topology of weak convergence,
		\item the functional $\pi \to  \int \omega d\pi$ is weakly continuous on $\Pi(M^X, M^Y)$ for every $\omega\in \Omega$,
		\item the set $\Pi_R(\mu,\nu)$ is nonempty for any $\mu\in M^X$, $\nu\in M^Y$.
	\end{enumerate}
\end{definition}

\begin{proposition}
	\label {weak regularity implies closedness}
	Let $R=(\Omega, M^X, M^Y)$ be weakly regular linear restriction.
	Then $\Pi_R(X\times Y)$ is closed, and $\Pi_R(\mu,\nu)$
	is compact for any pair of measures $(\mu, \nu)\in M^X\times M^Y$. 
	Both spaces,
	$\Pi_R(X\times Y)$ and $\Pi_R(\mu,\nu)$, appear to be Polish.
\end{proposition}
\begin{proof}
	It follows form the fact, that $\Pi(\mu, \nu)$
	is compact in the topology of weak convergence (see \cite{BogKol}).
\end{proof}

Let us provide an (obvious) condition on a linear restriction $R$,
which is sufficient for its weak regularity.

\begin{remark}
	\label {bounded continuous weak regularity}
	If $\Omega\subseteq C_b(X\times Y)$, $M^X\subseteq {\mathcal P}(X)$ and
	$M^Y\subseteq {\mathcal P}(Y)$ are weakly closed and nonempty, and
	$\mu\otimes\nu\in \Pi_R(\mu, \nu)$ for any pair of measures $\mu\in M^X$, $\nu\in M^Y$, then
	$R=(\Omega, M^X, M^Y)$ is a weakly regular restriction.
\end{remark}

Let us fix, in addition to the restriction $R$, some cost functional $C$. We define two sets of restricted optimal plans:

\begin{equation}
\Pi_R^{opt}(X\times Y):=\Big\{\pi\in \Pi_R(X\times Y)\colon\  C(\pi)\leq C(\gamma)\ \forall \gamma\in \Pi_R(\mathrm{Pr}(\pi))\Big\},
\end{equation}
\begin{equation}
\Pi_R^{opt}(\mu,\nu):=\Big\{\pi\in \Pi_R(\mu,\nu)\cap \Pi_R^{opt}(X\times Y)\Big\}.
\end{equation}

\begin{proposition}
	\label{compactness of optimal plans}
	If $R=(\Omega, M^X, M^Y)$ is weakly regular restriction, and $C$ is weakly l.s.c., then $\Pi^{opt}_R(\mu,\nu)$ is a compact subspace of ${\mathcal P}(X\times Y)$ (for any $\mu\in M^X$, $\nu \in M^Y$).
\end{proposition}
\begin{proof}
	By definition of lower semi-continuity of a functional, its lower level sets are closed. Thus the set $\{\pi\in \Pi_R(\mu,\nu)\colon\  C(\pi)\leq a\}$
	is closed for any $a\in \R$. Let $a=\inf\left\{C(\pi)\colon\
	\pi\in \Pi_R(\mu,\nu)\right\}$, then
	the corresponding level set is $\Pi_R^{opt}(\mu,\nu)$.
	Since it is a closed subset of a weakly compact Polish space $\Pi_R(\mu,\nu)$
	(by Proposition \ref{weak regularity implies closedness}), it is itself weakly compact and Polish.
\end{proof}

As we see in the next sections, the assumption of weak regularity is enough to prove the existence of a measurable selection of optimal transport plans in the restricted set. But for the decomposition result, analogous to Theorem \ref{1 main theorem}, we need to assume more.

By the end of the section we shall assume, that $R:=(\Omega, M^X, M^Y)$ is a linear restriction, $M^X$, $M^Y$ ergodic decomposable simplexes,
${\mathcal A}^0\subseteq {\mathcal A}$, ${\mathcal B}^0\subseteq {\mathcal B}$ are correspondent $H$-sufficient $\sigma$-algebras.

The following property of linear restriction is the key one for the formulation of the main decomposition result.
\begin{definition}
	\label{def ergodic restriction}
	A linear restriction $R:=(\Omega, M^X, M^Y)$ is called \textbf{ergodic decomposable} linear restriction,
	if there exists such an ergodic decomposable simplex $M\subseteq \mathcal P(X\times Y)$ 
	and its correspondent $H$-sufficient $\sigma$-algebra
	$({\mathcal A} \otimes {\mathcal B})^0\subseteq {\mathcal A} \otimes {\mathcal B}$, that
	\begin{itemize}
		\item $\Pi_R(X\times Y)\subseteq M$,
		\item ${\mathcal A}^0\otimes {\mathcal B}^0\subseteq({\mathcal A} \otimes {\mathcal B})^0$,
		\item $\mathrm {Pr}_X(\gamma)\in \partial_e(M^X)$,
		$\mathrm {Pr}_Y(\gamma)\in \partial_e(M^Y)$, $\forall \gamma\in \partial_e(M)$.
	\end{itemize}
\end{definition}

Let us assume, that $R:=(\Omega, M^X, M^Y)$ is an ergodic decomposable restriction,
$M\subseteq \mathcal P(X\times Y)$ is the correspondent ergodic decomposable simplex,
$Q_M$ and $({\mathcal A} \otimes {\mathcal B})^0\subseteq {\mathcal A} \otimes {\mathcal B}$ are the associated Markov kernel and $H$-sufficient $\sigma$-algebra.

Let us introduce a property of linear restriction, which assures coherency of simplexes of marginal measures and the additional restriction on transport plans.

\begin{definition}
	\label{def coherent linear restriction}
	A linear restriction $R:=(\Omega, M^X, M^Y)$ is called coherent for $M$,
	if the inclusion $\pi\in \Pi_R(X\times Y)$ implies that
	$\int_S \omega d\pi=0$ for any
	$\omega\in \Omega$, $S\in {({\mathcal A} \otimes {\mathcal B})}^0$.
\end{definition}

In practice, we shall use the following sufficient condition for coherency.
\begin{proposition}
	\label{sufficient condition for coherency}
	If there exists such a family $\{F_\alpha\}$ of maps
	$F_\alpha\colon\
	B(X\times Y, {\mathcal A}\otimes {\mathcal B})\to
	B(X\times Y, {\mathcal A}\otimes {\mathcal B})$, $\alpha\in \Lambda$, that
	\begin{enumerate}
		\item $F_\alpha(g f)=g F_\alpha(f)$ for any
		$g\in B(X\times Y, {({\mathcal A} \otimes {\mathcal B})}^0)$,
		$f\in B(X\times Y, {\mathcal A}\otimes {\mathcal B})$, $\alpha\in \Lambda$,
		\item linear restrictions $R=(M^X, M^Y, \Omega_{cont})$
		and $R_{meas}=(M^X, M^Y, \Omega_{meas})$ are equivalent,
		where $\Omega_{cont}:=\mathrm{span}\{f-F_\alpha(f)\colon\  f\in C_b(X\times Y),\
		\alpha\in \Lambda\}$,
		$\Omega_{meas}:=\mathrm{span}\{f-F_\alpha(f)\colon\  f
		\in B(X\times Y, {\mathcal A}\otimes {\mathcal B}),\ \alpha\in \Lambda\}$,
	\end{enumerate}
	then $R$ is a coherent linear restriction.
\end{proposition}
\begin{proof}
	We wish to prove that $\int I_S (f-F_\alpha(f))d\pi=0$ for any
	$f\in C_b(X\times Y)$, $S\in {({\mathcal A} \otimes {\mathcal B})}^0$,
	$\pi\in \Pi_R(X\times Y)$, $\alpha\in \Lambda$, where $I_S$ ---
	is the indicator function of $S$.
	Since $\Pi_R(X\times Y)=\Pi_{R_{meas}}(X\times Y)$,
	$\int (f-F_\alpha(f)) d\pi=0$ for any $\pi\in \Pi_R(X\times Y)$,
	$f\in B(X\times Y, {\mathcal A}\otimes {\mathcal B})$, $\alpha\in \Lambda$.
	The statement of interest follows from the fact that
	$\int I_S (f-F_\alpha(f))d\pi=\int (I_S f-F_\alpha(I_S f))d\pi$
	and $I_S f\in B(X\times Y, {\mathcal A}\otimes {\mathcal B})$ for every $f\in C_b(X\times Y)$, $S\in {({\mathcal A} \otimes {\mathcal B})}^0$, $\alpha\in \Lambda$.
\end{proof}

\begin{remark}
	\label{simplification of coherency}
	In the case $\Pi_R(\mu,\nu)=\Pi(\mu,\nu)\cap M$, $\forall (\mu,\nu)\in M^X\times M^Y$ for some ergodic decomposable restriction $R$, 
	it is true that $R$ is coherent for $M$. Indeed, by definition, $M$ is the set of all $Q_M$-invariant measures from ${\mathcal P}(X\times Y)$.
	It is enough to consider $F_1(f):=Q_M(f)$, $\Lambda=\{1\}$, $\tilde{\Omega}:=\{f-Q_M(f),\ \forall f\in C_b(X\times Y)\}$. 
	It follows, that the restriction $\tilde{R}:=(\tilde{\Omega}, M^X, M^Y)$ is coherent and equivalent to $R$.
\end{remark}

As we see later, the assumptions of coherency, ergodic decomposability, and weak regularity together imply the decomposition result for the associated Kantorovich problem.

Let $d\colon\  X\times X \to  \R$ be a given distance function,
$\mathrm{Dom}_Q\subseteq {\mathcal P}(X)$ be an ergodic decomposable simplex,
$R=(\Omega, \mathrm{Dom}_Q, \mathrm{Dom}_Q)$ be a linear restriction.
Define for each number $p\in [1,+\infty)$ a function
$W_p^R\colon\ \mathrm{Dom}_Q\times \mathrm{Dom}_Q \to  [0,\infty]$ by the formula:
\begin{equation}
\label{generalized kantorovich metric}
W_p^R(\mu,\nu):=\inf\left\{\left(\int d^p(x, y)d\pi\right)^\frac{1}{p}\colon\  \pi\in \Pi_R(\mu, \nu)\right\}.
\end{equation}
This function does not satisfy distance axioms in general. This motivates the following, another one, assumption about liner restriction, which assures  $W_p^R$ to be a distance function.

\begin{definition}
	\label{def metric linear restriction}
	A linear restriction $R=(\Omega_Q,\mathrm{Dom}_Q, \mathrm{Dom}_Q)$ is a \textbf{geometric} linear restriction for $Q$ iff the simplex
	$\mathrm{Dom}_Q$ is weakly closed, $\Omega_Q\subseteq C_b(X\times Y)$, and
	for every $\omega\in \Omega_Q$ it is true that:
	\begin{enumerate}
		\item $((Id, Id)_\#\mu) (\omega)=0$ \ $\forall \mu\in \mathrm{Dom}_Q$,
		\item $(\mu\otimes \nu)(\omega)=0$ \ $\forall \mu, \nu\in \mathrm{Dom}_Q$,
		\item $\pi(\omega)=0\implies \pi^T(\omega)=0$ \
		$\forall \pi\in \Pi(\mu,\nu), \forall \mu, \nu\in \mathrm{Dom}_Q$,
		where $\pi^T$ is defined as follows:
		$$
		\int f(x,y) d\pi^T=\int f(y,x) d\pi \ \forall f\in C_b(X\times Y).
		$$
	\end{enumerate}
\end{definition}

\begin{proposition}
	\label{geometricity implies distance function}
	If $R$ is a geometric linear restriction,
	then $W_p^R$ is an extended distance function.
\end{proposition}
\begin{proof}
	Since $\pi=(Id,Id)_{\#}\mu\in \Pi_R(\mu,\mu)$, $\int d^p d\pi=0$
	for this $\pi$, and $W_p^R(\mu,\mu)=0$.
	By the same reason ($d^p=0$ only on the diagonal $\{(x,x)\}$, 
	and only plans of the form $(Id,Id)_{\#}\mu$ are concentrated on it)
	$W_p^R(\mu,\nu)\neq 0$ if $\mu\neq \nu$.
	Function $W_p^R$ is symmetric, because the inclusion
	$\pi \in \Pi_R(\mu,\nu)$ implies $\pi^T \in \Pi_R(\nu,\mu)$ and $\int d^p d\pi=\int d^p d\pi^T$.
	Triangle inequality can be proved using the standard technique (see, for example, \cite{AmbGigli1}, Theorem 2.2)
	with the use of the special version of the ``gluing'' lemma, that is formulated below.
\end{proof}

\begin{lemma}[version of gluing]
	\label{gluing lemma}
	Let $R=(\Omega_Q,\mathrm{Dom}_Q, \mathrm{Dom}_Q)$ be a geometric linear restriction.
	Then for every measure $\mu_1, \mu_2, \mu_3\in \mathrm{Dom}_Q$, $\pi_{12}\in \Pi_R(\mu_1,\mu_2)$, $\pi_{23}\in \Pi_R(\mu_2,\mu_3)$
	there is a measure $\gamma\in \mathcal{P}(X\times X\times X)$, such that $(\Pr_{12})_{\#}(\gamma)=\pi_{12}$, $(\Pr_{23})_{\#}(\gamma)=\pi_{23}$,
	$(\Pr_{13})_{\#}(\gamma)\in \Pi_R(\mu_1,\mu_3)$.
\end{lemma}
\begin{proof}
	Let us modify a proof of the standard gluing lemma (see Lemma 2.2. of \cite{Chodosh}). 
	Define subspace $V\subset C_b(X\times X\times X)$ as follows:
	$$
	V:=\Big\{f_{12}(x_1,x_2)+f_{23}(x_2,x_3)+\omega_{13}(x_1,x_3)\colon\ f_{12}, f_{23}\in C_b(X\times X),\ \omega_{13}\in \Omega_Q\Big\}.
	$$
	Let $F(f_{12}+f_{23}+\omega_{13}):= \pi_{12}(f_{12})+\pi_{23}(f_{23})$. 
	Check that $F$ is well-defined on $V$.
	Consider two representations of some element of $V$: $f_{12}+f_{23}+\omega_{13}=\tilde{f}_{12}+\tilde{f}_{23}+\tilde{\omega}_{13}$.
	Note, that $\omega_{13}(x_1,x_3)-\tilde{\omega}_{13}(x_1,x_3)=\omega_1(x_1)+\omega_3(x_3)$ for some $\omega_1, \omega_3\in \Omega_Q$.
	Then $f_{12}-\tilde{f}_{12}+\omega_1=\tilde{f}_{23}-f_{23}-\omega_3$, and therefore both parts of the equality depends only on $x_2$. 
	Since
	$$
	\pi_{12}(f_{12}-\tilde{f_{12}}+\omega_1)=\mu_2(f_{12}-\tilde{f_{12}})=\mu_2(\tilde{f}_{23}-f_{23})=\pi_{23}(\tilde{f}_{23}-f_{23}-\omega_3),
	$$
	the map $F$ is well-defined on $V$. 
	It is easy to check, that $F$ is a bounded positive linear functional.
	The correspondent version of Hahn-Banach theorem (Theorem 1.25 of \cite{Aliprantis}) states, 
	that such a functional can be extended to a positive bounded functional on the entire $C_b(X\times X\times X)$.
	Since on the subspaces $C_b(X_k)$ value of $F$ coincides with the integration w.r.t. measures $\mu_k$, 
	one can apply Rietz theorem (Theorem 7.10.6 in \cite{Bogachev}) to the extension of $F$. 
	The resulting measure will satisfy all the required properties from the statement of the lemma we are proving.
\end{proof}

\begin{proposition}
	\label{geometricity implies weak regularity}
	Geometric linear restriction is weakly regular.
\end{proposition}
\begin{proof}
	Since $(\mu\otimes\nu)(\omega)=0$, $\mu\otimes\nu\in \Pi_R(\mu,\nu)$ for any
	$\mu,\nu\in \mathrm{Dom}_Q$, $\mathrm{Dom}_Q$ is weakly closed and
	$\Omega\subseteq C_b(X\times Y)$, it follows that $R=(\Omega, \mathrm{Dom}_Q, \mathrm{Dom}_Q)$ is weakly regular (see Remark \ref{bounded continuous weak regularity}).
\end{proof}

Let us check that for our main examples (\ref{example invariant problem 1} and \ref{example invariant problem 2} of invariant and stationary measures respectively) it is possible to define additional linear restrictions, such that all (or most of) properties introduced above are satisfied.
\begin{example}[Main example]
	\label{example invariant problem 2}
	Let us consider the linear restriction: $R=(\Omega, \mathrm{Dom}, \mathrm{Dom})$,
	where $\mathrm{Dom}$ is a simplex of all invariant measures as in Example \ref{example invariant problem 1},
	$$
	\Omega:=\mathrm{span}(\{f-f\circ g\colon\  \forall f\in C_b(X\times X), \forall g\in \mathbb G \}).
	$$
It corresponds to a restriction on transport plans to be invariant w.r.t. a diagonal action of $\mathbb G$ on $X\times X$ (see \cite{Zaev}, Prop. 5.1 for the proof): $g(x,y)=(g(x), g(y))$.

	Let $M=\mathcal P_{\mathbb G}(X\times Y)$ be a simplex of all invariant measures on $X\times Y$ w.r.t. this action,	$({\mathcal A} \otimes {\mathcal B})^{inv}$ be a $\sigma$-algebra of all invariant Borel subsets on $X\times Y$.
	It is clear, that $\Pi_R(X\times Y)\subseteq M$, ${\mathcal A}^{inv} \otimes {\mathcal B}^{inv}\subseteq ({\mathcal A} \otimes {\mathcal B})^{inv}$ 
	and $({\mathcal A} \otimes {\mathcal B})^{inv}$ is the $H$-sufficient subalgebra associated to $M$. 
	To show \textit{ergodic decomposability} of $R$ (Definition \ref{def ergodic restriction}), 
	we need to check, that every ergodic measure from $P_{\mathbb G}(X\times Y)$ 
	has ergodic marginals, which belong to ${\mathcal P}_{\mathbb G}(X)$. It can be easily shown by contradiction.
	
	Let us prove that the restriction is \textit{geometric} (Definition
	\ref{def metric linear restriction}):
	\begin{enumerate}
		\item the set ${\mathcal P}_{\mathbb G}(X)$ is closed and $\Omega\subset C_b(X\times X)$,
		\item $(Id,Id)_\#\mu(f(x,y)-f(g(x),g(y)))=\mu(f(x,x)-f(g(x,x)))=\mu(f(x,x))-g_\#\mu(f(x,x))=0$ if $\mu\in Dom={\mathcal P}_{\mathbb G}(X)$, $f\in C_b(X\times X)$, $g\in \mathbb G$, $\mu(f):=\int f d\mu$,
		\item note that
		\begin{multline*}
		(\mu\otimes \nu) (f(x,y)-f(g(x),g(y))=\int \left(\int f(x,y)-f(g(x),y) d\mu(x) \right)d\nu(y) +\\
		+\int \left(\int f(g(x),y)-f(g(x),g(y)) d\nu(y) \right)d\mu(x)=0,\ \forall \mu,\nu\in \mathrm{Dom}={\mathcal P}_{\mathbb G}(X),\ f\in C_b(X\times X),\ g\in \mathbb G,
		\end{multline*}
		\item since the inclusion
		$f(x,y)-f(g(x), g(y))\in \Omega$ implies $f(y,x)-f(g(y),\ g(x))\in \Omega$ for any
		$f\in C_b(X\times X)$, the last requirement from the definition of geometricity is satisfied.
	\end{enumerate}
	We also need to check, that the restriction is \textit{coherent} w.r.t. $M$
	(Definition \ref{def coherent linear restriction}). 
	Note, that $\Pi_R(\mu,\nu)=\Pi(\mu,\nu)\cap {\mathcal P}(X\times Y)$, $\forall (\mu,\nu)\in M^X\times M^Y$, 
	and use Remark \ref{simplification of coherency}.
\end{example}

\begin{remark}
	For the example with a simplex of invariant measures $\mathrm{Dom}$ (Example \ref{example invariant problem 1}) it is also possible to consider other  meaningful and well-behaved linear restrictions. 
	Assume that there is a given action of group ${\mathbb G}$ on $X$. 
	One can consider the action of direct product of groups ${\mathbb G}\oplus {\mathbb G}$ of $X\times X$ defined in the natural way: $(g_1,g_2)(x_1, x_2):=(g_1(x_1), g_2(x_2))$.
	
	Let us fix a subgroup ${\mathbb H}\subseteq {\mathbb G}\oplus {\mathbb G}$ with the induced action on $X\times X$. If the projections of the subgroup on the first and second components coincide with ${\mathbb G}$, the associated restriction $R=(\Omega, \mathrm{Dom}, \mathrm{Dom})$ with
	$$
	\Omega:=\mathrm{span}(\{f-f\circ h\colon\  \forall f\in C_b(X\times X), \forall h\in {\mathbb H} \}).
	$$
	has the properties of ergodic decomposability and coherency, defined in this section. 
	It can be checked in the same way we do for the diagonal action of a group.
\end{remark}

\begin{example}[Discrete-time Markov process]
	Let simplex $\mathrm{Dom}={\mathcal P}_{Q}(X)$ be defines as in Example \ref{example markov process}, and let ${\mathcal A}^0$ be its associated $\sigma$-subalgebra.
	Consider Markov transition kernel $Q_M$ from ${\mathcal A}\otimes {\mathcal A}$ to ${\mathcal A}^0\otimes {\mathcal A}^0$
	that is defined by the formula
	$$
	Q_M(f)(x,y):=\int f(\tilde{x}, \tilde{y}) dQ^x(\tilde{x})\otimes Q^y(\tilde{y}).
	$$ 
	It can be checked, that $Q_M(fQ_M(g))=Q_M(f)Q_M(g)$, and, therefore, by Theorem \ref{equivalent definitions of ergodic decomposable simplex}
	there is an associated ergodic decomposable simplex $M$.
	Let us consider the following linear restriction 
	$R=(\Omega, \mathrm{Dom}, \mathrm{Dom})$,
	$$
	\Omega:=\mathrm{span}(\{f-Q_M(f)\colon\  \forall f\in C_b(X\times X)\}).
	$$
	It is ergodic decomposable and coherent w.r.t. $M$ 
	(the arguments are analogous to the ones from the previous example).
\end{example}

\section{Measurable selection of optimal transport plans}
\sectionmark{Measurable selection of optimal transport plans}

The goal of this section is to prove the existence of a measurable map $f: M^X\times M^Y \rightarrow \Pi^{opt}_R(X\times Y)$, such that $\mathrm{Pr}(f(\mu,\nu))=(\mu,\nu)$ (recall, that $\Pr(\pi):=(\Pr_X(\pi),\Pr_X(\pi))$), under the assumptions of weak regularity of $R$ and lower semi-continuity of cost functional. Existence of such a map is required in the proof of the decomposition result, which we shall formulate in the next section, and it seems to be of interest itself.

In the Kantorovich problem without additional linear restriction this fact is well-known (see Corollary 5.22 of \cite{Villani1}). Its proof relies on the sufficiency result: $c$-cyclical monotonicity of a support of a transport plan implies its optimality. There is no known analogue of this result in the case of the problem with additional restriction. Thus we need to invent a more direct way to prove measurability.

Our main tool is the following theorem of Rieder:

\begin{theorem}[Rieder \cite{Rieder}, Th. 4.1, Cor. 4.3]
	\label{measurable optimization}
	Let $(\Theta,\calA)$, $(S,\calB)$ be Polish spaces with Borel $\sigma$-algebras, $D\subseteq \Theta\times S$, $u: D\rightarrow \R\cup \{\pm \infty\}$,
	$$
	L_c:=\{(\theta,s)\in D: u(\theta,s)\leq c\},\ c\in \R,
	$$
	$L_{+\infty}:=D$. If $\forall c\in (-\infty, +\infty]$, $L_c\in \calA\otimes \calB$, and $\forall c\in \R$
	$$
	L_c(\theta):=L_c \cap \{(\theta,s): s\in S\}
	$$
	is compact for any $\theta\in \Pr_\Theta(D)$, then there exists a measurable function $f: \Pr_\Theta(D) \rightarrow S$, such that 
	$$
	u(\theta,f(\theta))=\inf_{s\in D(\theta)} u(\theta,s),\ \forall \theta\in \Pr_\Theta(D)
	$$
	where $D(\theta):=\Pr_\Theta^{-1}(\theta)\cap D$.
\end{theorem}

Let $\Theta:=M^X\times M^Y$, $S:=\Pi_R(X\times Y)$, be equipped with the topology of weak convergence and corresponding Borel $\sigma$-algebras. Both $\Theta$ and $S$ are Polish spaces, and due to weak regularity of $R$, $S$ is closed. Let $D:=\{(\mu,\nu,\pi): \pi\in \Pi_R(\mu,\nu),\ (\mu, \nu)\in M^X\times M^Y\}$, $u((\mu,\nu),\pi):=Cost(\pi)$. Then
$$
L_c=\{(\mu, \nu, \pi): Cost(\pi)\leq c,\ \pi\in \Pi_R(\mu,\nu),\ (\mu, \nu)\in M^X\times M^Y\},\ c\in (-\infty, +\infty],
$$
$$
L_c(\mu,\nu)=\{\pi: Cost(\pi)\leq c,\ \pi\in \Pi_R(\mu,\nu)\},\ c\in \R
$$
To apply the theorem above, we have to prove that $L_c$ is Borel, and that $L_c(\mu,\nu)$ is compact.

We are going to use the following proposition (Himmelberg \cite{Himmelberg}, Theorem 3.5).
\begin{proposition}
	\label{graph measurability effros}
	Let $(\Theta,\calA)$, $(Y,\calB)$ be separable metrizable spaces with Borel $\sigma$-algebras. Let $T:\Theta \rightarrow 2^S$ be a map with values in closed subsets of $S$, such that for every closed subset $V\subseteq Y$
	$$
	\{\theta: T(\theta)\cap V \neq \emptyset\}\in \calA
	$$
	Then the graph of $T$:
	$$
	\{(\theta,s): \theta\in \Theta,\ s\in T(\theta)\}
	$$
	is in $\calA \otimes \calB$.
\end{proposition}

\begin{lemma}
	\label{graph is a borel function}
	If $R$ is weakly regular and $Cost: \Pi_R(X\times Y) \rightarrow \R$ is l.s.c., then the set 
	$$
	L_c:=\{(\mu, \nu, \pi): Cost(\pi)\leq c,\ \pi\in \Pi_R(\mu,\nu),\ (\mu, \nu)\in M^X\times M^Y\},\ c\in (-\infty, +\infty],
	$$ 
	is an element of $\sigma$-algebra $\mathrm{Bor}(M^X\times M^Y)\otimes \mathrm{Bor}(\Pi_R(X\times Y))$.
\end{lemma}
\begin{proof}
	By definition of weak regularity, $\Pi_R(X\times Y)$ is a Polish space, $\Pi_R(\mu,\nu)$ is Polish and compact. Since $Cost$ is an l.s.c. functional, its lower level sets $\Pi_c:=\{\pi\in \calP(X\times Y): Cost(\pi)\leq C\}$ are closed. Since $L_c(\mu,\nu)=\Pi_R(\mu,\nu)\cap \Pi_c$, $L_c(\mu,\nu)$ is compact. It should be noted, that $L_c(\mu,\nu)$ can be empty, but empty set is compact, hence there is no contradiction here.
	
	Let $(\Omega,\calA):=(M^X\times M^Y, \mathrm{Bor}(M^X\times M^Y))$, $(S,\calB):= (\Pi_R(X\times Y),\mathrm{Bor}(\Pi_R(X\times Y)))$, $T: M^X\times M^Y \rightarrow 2^{\Pi_R(X\times Y)}$, $T(\mu,\nu)=L_c(\mu,\nu)$. Then $L_c=\{(\mu,\nu,\pi): \pi\in L_c(\mu,\nu),\ (\mu, \nu)\in M^X\times M^Y\}$ is the graph of $T$. Note, that $T$ has compact (hence closed) values.
	
	Fix an arbitrary closed set $V\subseteq \Pi_R(X\times Y)$. Let us show that the set
	$$
	T^{-1}(V):=\{(\mu,\nu): L_c(\mu,\nu)\cap V \neq \emptyset\}
	$$
	is closed, hence Borel.
	If the set is empty, it is trivially true. Assume it is not. Let $(\mu_n,\nu_n)\in T^{-1}(V)$, $n\in \N$, be a sequence converging to $(\mu, \nu)$ in $M^X\times M^Y$. Then $\mu_n\rightarrow \mu$, $\nu_n\rightarrow \nu$ weakly, and the sequences $(\mu_n)$, $(\nu_n)$ are tight. But for each $(\mu_n, \nu_n)$ there exists a $\pi_n\in L_c(\mu_n,\nu_n)\cap V$. It follows, that
	$\Pr(\pi_n)=(\mu_n,\nu_n)$, and the sequence $(\pi_n)$ is tight (by considering products of compact sets). Let $(\pi_{n_k})$ be a weakly convergent subsequence with limit $\pi$. Since $V$ is closed, it is sequentially closed, and $\pi\in V$.
	Due to weak continuity of $\Pr$ and weak l.s.c. of $Cost$, it is clear, that $\Pr(\pi)=(\mu,\nu)$, $Cost(\pi)\leq c$. It follows by weak regularity of $R$, that the functional $\pi\rightarrow \int \omega d\pi$ is weakly continuous for any $\omega \in \Omega$. Hence $\int \omega d\pi=0$, and $\pi\in L_c(\mu,\nu)$.
	
	We obtain, that for $(\mu,\nu)$ there is a $\pi\in L_c(\mu,\nu)\cap V$, hence $(\mu,\nu)\in T^{-1}(V)$, and $T^{-1}(V)$ sequentially closed. Since $M^X\times M^Y$ is metrizable, sequential closeness implies closeness, and $T^{-1}(V)$ is closed, hence Borel. By Proposition \ref{graph measurability effros}, the graph of $T$, which coincides with $L_c$, is a Borel set. It concludes the proof.
\end{proof}

We are ready to formulate and prove the main result of this section.
\begin{theorem}
	\label{measurable optimal kernel theorem}
	Let $X$, $Y$ be Polish spaces with Borel $\sigma$-algebras.
	If $R=(\Omega, M^X, M^Y)$ is a weakly regular linear restriction, $Cost$ is a weakly l.s.c. cost functional, then there exists a measurable map $f: M^X\times M^Y \rightarrow \Pi^{opt}_R(X\times Y)$, such that $\mathrm{Pr}(f(\mu,\nu))=(\mu,\nu)$.
\end{theorem}

\begin{proof}
	Let us apply Theorem \ref{measurable optimization}.
	Recall, that $\Theta:=M^X\times M^Y$, $S:=\Pi_R(X\times Y)$ are Polish spaces with Borel $\sigma$-algebras, $D:=\{(\mu,\nu,\pi): \pi\in \Pi_R(\mu,\nu),\ (\mu,\nu)\in M^X\times M^Y\}$, $u((\mu,\nu),\pi):=Cost(\pi)$,
	$
	L_c=\{(\mu, \nu, \pi): Cost(\pi)\leq c,\ \pi\in \Pi_R(\mu,\nu),\ (\mu,\nu)\in M^X\times M^Y\}
	$ is Borel for every $c\in (-\infty,+\infty]$ by Lemma \ref{graph is a borel function},
	$
	L_c(\mu,\nu)=\{\pi: Cost(\pi)\leq c,\ \pi\in \Pi_R(\mu,\nu)\}
	$ is compact for every $c\in \R$ by weak regularity of $R$ and lower semi-continuity of $Cost$. It follows directly from the application of Theorem \ref{measurable optimization}, that there exists a Borel function $f: M^X\times M^Y \rightarrow \Pi_R(X\times Y)$, such that
	$$
	Cost(f(\mu,\nu))=\inf_{\pi\in \Pi_R(\mu,\nu)} Cost(\pi),
	$$
	which implies, that $f(\mu,\nu)\in \Pi^{opt}_R(\mu,\nu)$ for any $(\mu,\nu)\in M^X\times M^Y$.
\end{proof}

\section{Decomposition of Kantorovich problem}
\sectionmark{Decomposition of Kantorovich problem}

In this section we formulate and prove an analogue of Theorem \ref{1 main theorem} for the restricted Kantorovich problem under the assumptions of weak regularity, ergodic decomposability, and coherency of linear restrictions.

Let $(X,{\mathcal A})$, $(Y,{\mathcal B})$ be two Polish spaces with Borel $\sigma$-algebras, $M^X\subseteq {\mathcal P}(X)$ and $M^Y\subseteq {\mathcal P}(Y)$
be two closed ergodic decomposable simplexes. Denote via $Q_X$ and $Q_Y$ the  Markov kernels on $(X,{\mathcal A})$, $(Y,{\mathcal B})$, associated with $M^X$ and $M^Y$, such that
$\partial_e(M^X)=\{Q_X^x\}$, $\partial_e(M^Y)=\{Q_Y^y\}$.
If ${\mathcal A}^0\subseteq {\mathcal A}$ and ${\mathcal B}^0\subseteq {\mathcal B}$ be their corresponding $H$-sufficient subalgebras,
then, by definition, $Q_X$ and $Q_Y$ are decompositions for triples $({\mathcal A}, {\mathcal A}^0, M^X)$ and $({\mathcal B}, {\mathcal B}^0, M^Y)$
respectively. We shall use the notation: $\xi(x):=Q_X^x$, $\eta(y):=Q_Y^y$. Note,
that the maps $\xi\colon\  X\to  \partial_e(M^X)$,
$\eta\colon\  Y\to  \partial_e(M^Y)$ are
${\mathcal A}^0$- and ${\mathcal B}^0$-measurable respectively. 

Let $R=(\Omega, M^X, M^Y)$ be
\textit{weakly regular}, \textit{ergodic decomposable} and \textit{coherent} w.r.t. $M$ linear restriction, 
where $M\subseteq \mathcal P(X\times Y)$ is the corresponding ergodic decomposable simplex, 
$Q_M$ and $({\mathcal A} \otimes {\mathcal B})^0\subseteq {\mathcal A} \otimes {\mathcal B}$ are the associated Markov kernel and $H$-sufficient $\sigma$-algebra.

\begin{definition}
	\label{definition tilde pi}
	Denote by $\tilde{\Pi}(\tilde{\mu},\tilde{\nu})$ the set of all probability measures $\tilde{\pi}\in {\mathcal P}(X\times Y, ({\mathcal A} \otimes {\mathcal B})^0)$, such that $\mathrm{Pr_X}(\tilde{\pi})=\tilde{\mu}$, $\mathrm{Pr_Y}(\tilde{\pi})=\tilde{\nu})$.
\end{definition}

\begin{definition}
	Let $\mu\in M^X$, $\nu\in M^Y$.
	Define the set $\Theta(R, \mu, \nu)$ as the set of all pairs $(\pi, Q_{\pi})$, where
	$\pi\in \Pi_R(\mu,\nu)$, and $Q_\pi$ is such a Markov transition kernel from $(X\times Y, {\mathcal A}\otimes {\mathcal B})$ 
	to $(X\times Y,({\mathcal A} \otimes {\mathcal B})^0)$, that
	$Q_{\pi}(\tilde{\pi})=\pi$ and $Q_\pi^{(x,y)}\in\Pi_R(Q_X^x, Q_Y^y)$ 
	for $\tilde{\pi}$-a.e. $(x,y)\in X\times Y$. Here $\tilde{\pi}$ is the restriction of measure $\pi$ from ${\mathcal A}\otimes {\mathcal B}$ 
	to $({\mathcal A} \otimes {\mathcal B})^0$.
\end{definition}

In the following Lemma we use properties of ergodic decomposability and coherency of the linear restriction.

\begin{lemma}
	\label{existence of a decomposition for a plan}
	Under the assumptions made above, for any $\pi\in \Pi_R(\mu,\nu)$ there exists a Markov transition kernel $Q_\pi$, such that $(\pi,Q_{\pi})\in\Theta(R,\mu,\nu)$.
\end{lemma}
\begin{proof}
	For each measure $\pi \in \Pi(\mu,\nu)$ define
	measure $\tilde{\pi}$ as the restriction of $\pi$ on $({\mathcal A} \otimes {\mathcal B})^0$.
	Its marginals, $\mathrm{Pr}_X(\tilde{\pi})$ and $\mathrm{Pr}_Y(\tilde{\pi})$,
	are the restrictions of $\mu$ and $\nu$ to
	subalgebras ${\mathcal A}^0$ and ${\mathcal B}^0$ respectively.
	Hence, the inclusion
	$\pi \in \Pi(\mu,\nu)$ implies that $\tilde{\pi}\in \tilde{\Pi}(\tilde{\mu},\tilde{\nu})$
	(see Definition \ref{definition tilde pi}).
	
	Let $Q_\pi:=Q_M$, where $Q_M$ is the Markov kernel from the definition of ergodic decomposability of restriction $R$. 
	This property implies $Q(\tilde{\pi})=\pi$ and $Q_\pi^{(x,y)} \in {\mathcal P}(X\times Y)$. Moreover,
	$\mathrm{Pr}(Q_\pi^{(x,y)})=(\xi(x),\eta(y))$ for $\tilde{\pi}$-a.e. $(x,y)$.
	
	Let us check that $\tilde{\pi}$-a.e. $Q_\pi^{(x,y)}(\omega)=0$
	for each function $\omega\in \Omega$. Let us use the following equality:
	\begin{equation}
	\label {Q-pi from decomposition of pi}
	\int h dQ^{(x,y)}_\pi=\mathbb E_\pi(h|({\mathcal A} \otimes {\mathcal B})^0)\ \pi\text{-a.e.}
	\ \forall h\in B(X\times Y,{\mathcal A}\otimes {\mathcal B}).
	\end{equation}
	By coherency of $R$ and the definition of conditional expectation, we obtain:
	$$
	\int_S \left(\int \omega dQ^{(x,y)}_\pi\right) d\tilde{\pi}= \int_S \omega d\pi = 0,\ \forall S\in ({\mathcal A} \otimes {\mathcal B})^0.
	$$
	Hence,
	$\int \omega dQ_\pi^{(x,y)}=0$ for $\tilde{\pi}$-a.e. $(x,y)$.
	
	As we just proved, $Q_\pi^{(x,y)}\in \Pi_R(\xi(x),\eta(y))$
	for $\tilde{\pi}$-a.e. $(x,y)$. Since the map $(x,y)\to  Q_\pi^{(x,y)}$ is $({\mathcal A} \otimes {\mathcal B})^0$-measurable,
	$Q_\pi$ is a Markov transition kernel from $(X\times Y,{\mathcal A}\otimes{\mathcal B})$
	to $(X\times Y,({\mathcal A} \otimes {\mathcal B})^0)$.
	
	It follows from (\ref{Q-pi from decomposition of pi}), that 
	$$
	\int Q_\pi(f)d\tilde{\pi}=\int f d\pi,\ \forall f\in B(X\times Y, {\mathcal A}\otimes {\mathcal B}),
	$$
	which implies $(\pi, Q_\pi)\in \Theta(R,\mu,\nu)$ and concludes the proof.
\end{proof}

Let $c\colon\ X\times Y \to  \R$ be such a measurable function,
that the functional $Cost\colon\ {\mathcal P}(X\times Y)\to  \R \cup \{+\infty\}$,
being defined as $Cost(\pi)=\int c d\pi$,
appears to be lower semi-continuous w.r.t. topology of weak convergence.

\begin{lemma}
	\label{existence of optimal Markov kernel}
	Under the assumptions made above, there exists a Markov transition kernel
	$Q_{opt}$ from $(X\times Y, {\mathcal A}\otimes {\mathcal B})$
	to $(X\times Y, ({\mathcal A} \otimes {\mathcal B})^0)$, such that
	\begin{enumerate}
		\item $\int c dQ_{opt}^{(x,y)} = \inf
		\left\{\int c d\pi\colon\  \pi \in \Pi_R(\xi(x), \eta(y))\right\}\ \forall (x,y)\in X\times Y$,
		\item for any pair $(\pi, Q_{\pi})\in \Theta(R,\mu,\nu)$,
		$\int c dQ_{opt}(\tilde{\pi})\leq \int c dQ_{\pi}(\tilde{\pi})$,
		where $\tilde{\pi}$ is the restriction of $\pi$ from ${\mathcal A}\otimes {\mathcal B}$
		to $({\mathcal A} \otimes {\mathcal B})^0$,
		\item for each measure $\tilde{\pi}
		\in \tilde{\Pi}(\tilde{\mu}, \tilde{\nu})$, $(Q_{opt}(\tilde{\pi}),
		Q_{opt})\in \Theta(R,\mu,\nu)$, where $\tilde{\mu}$ is the restriction of $\mu$ from ${\mathcal A}$ to ${\mathcal A}^0$, $\tilde{\nu}$ is the restriction of $\nu$ from ${\mathcal B}$ to ${\mathcal B}^0$.
	\end{enumerate}
\end{lemma}
\begin{proof}
	By Theorem \ref{measurable optimal kernel theorem}, there exists a measurable map $f\colon\  M^X\times M^Y \to  \Pi^{opt}_R(X\times Y)$, such that $\mathrm{Pr}(f(\mu,\nu))=(\mu,\nu)$. Let $Q_{opt}^{(x,y)}:=f(\xi(x),\eta(y))$.
	Since $x\to  \xi(x)$ and $y\to  \eta(y)$ 
	are measurable maps w.r.t. algebras ${\mathcal A}^0$ and ${\mathcal B}^0$ respectively,
	$Q_{opt}$ is in fact a Markov transition kernel from $(X\times Y, {\mathcal A}\otimes {\mathcal B})$ to
	$(X\times Y, {\mathcal A}^0\otimes {\mathcal B}^0)$ and, consequently, a Markov transition kernel to $({\mathcal A} \otimes {\mathcal B})^0$.
	It follows by definition, that
	$$
	\int c dQ_{opt}^{(x,y)} = \inf \left\{\int c d\gamma\colon\  \gamma \in \Pi_R(\xi(x), \eta(y))\right\}.
	$$
	Let $(\pi, Q_{\pi})\in \Theta(R,\mu,\nu)$. Then
	$\int c dQ_{\pi}(\tilde{\pi}):=\int \left(\int c dQ_{\pi}^{(x,y)}\right) d\tilde{\pi}$.
	Since $Q_{\pi}^{(x,y)}\in \Pi_R(\xi(x), \eta(y))$ $\tilde{\pi}$-a.e.,
	$\int c dQ_{\pi}^{(x,y)}\geq \int c dQ_{opt}^{(x,y)}$ $\tilde{\pi}$-a.e.,
	which implies $\int c dQ_{\pi}(\tilde{\pi})\geq \int c dQ_{opt}(\tilde{\pi})$.
	
	Let $\tilde{\pi}\in \tilde{\Pi}(\tilde{\mu}, \tilde{\nu})$,
	where $\tilde{\mu}$, $\tilde{\nu}$ are the restrictions of measures $\mu\in M^X$, $\nu\in M^Y$
	to ${\mathcal A}^0$, ${\mathcal B}^0$ respectively.
	By definition of Markov transition kernel, $Q_{opt}(\tilde{\pi})\in {\mathcal P}(X\times Y)$.
	Let us check that $Q_{opt}(\tilde{\pi})\in \Pi_R(\mu,\nu)$.
	Since $Q_{opt}^{(x,y)}\in \Pi_R(\xi(x), \eta(y))$
	for every pair $(x,y)\in X\times Y$, it can be shown,
	that $\mathrm{Pr}(Q_{opt}(\tilde{\pi}))=(\mu,\nu)$.
	Let us provide an argument for the first marginal:
	\begin{multline}
	\int \int f(\tilde{x})dQ_{opt}^{(x,y)}(\tilde{x},\tilde{y}) d\tilde{\pi}(x,y)=\int \int f(\tilde{x})dQ_X^x(\tilde{x}) d\tilde{\pi}(x,y)=\\
	=\int \int f(\tilde{x})dQ_X^x(\tilde{x}) d\tilde{\mu}(x)=\int f(x) d\mu,\ \forall f\in B(X,{\mathcal A}),
	\end{multline}
	where the last equality follows from the fact, that
	$Q_X$ is a decomposition for $({\mathcal A}, {\mathcal A}^0, M^X)$. Analogously, it can be checked for the second marginal. 
	It is clear, that if $\int \omega dQ_{opt}^{(x,y)}=0$ for any
	$(x,y)$, it implies the equality: $\int \int \omega dQ_{opt}^{(x,y)} d\tilde{\pi}=0$.
	Thus $Q_{opt}(\tilde{\pi})\in \Pi_R(\mu,\nu)$,
	and, hence, there exists a measure $\pi\in \Pi_R(\mu,\nu)$, such that
	$\pi=Q_{opt}(\tilde{\pi})\in \Pi_R(\mu,\nu)$.
\end{proof}

Let us formulate and prove the main result.

\begin{theorem}[Main theorem]
	\label{main theorem}
	Let $(X, {\mathcal A})$, $(Y, {\mathcal B})$ be two Polish spaces with Borel
	$\sigma$-algebras; $c\colon\ X\times Y \to  \R$ be such a measurable function, that functional $Cost\colon\ {\mathcal P}(X\times Y)\to   \R \cup \{+\infty\}$,
	defined as $Cost(\pi)=\int c d\pi$, appears to be a lower semi-continuous functional w.r.t. topology of weak convergence; $M^X$, $M^Y$ be two ergodic decomposable simplexes,
	${\mathcal A}^0\subseteq {\mathcal A}$, ${\mathcal B}^0\subseteq {\mathcal B}$ be corresponding $\sigma$-subalgebras; $R=(\Omega, M^X, M^Y)$ be a
	\textit{weakly regular}, \textit{ergodic decomposable}, and \textit{coherent} w.r.t. $M$ linear restriction, 
	where $M\subseteq \mathcal P(X\times Y)$ is the associated ergodic decomposable simplex,
	$Q_M$ and $({\mathcal A} \otimes {\mathcal B})^0\subseteq {\mathcal A} \otimes {\mathcal B}$ are the associated Markov kernel and $H$-sufficient $\sigma$-algebra.
	Then
	\begin{multline}
	\inf\left\{\int c d\pi\colon\  \pi \in \Pi_R(\mu, \nu)\right\}=\inf\left\{\int Q_{opt}(c) d\tilde{\pi}\colon\  \tilde{\pi}\in \tilde{\Pi}(\tilde{\mu},\tilde{\nu})\right\}=\\
	=\inf\left\{\int \inf \left\{\int c d\pi\colon\  \pi \in \Pi_R(\xi(x), \eta(y))\right\} d\tilde{\pi}\colon\  \tilde{\pi}\in \tilde{\Pi}(\tilde{\mu},\tilde{\nu})\right\}.
	\end{multline}
\end{theorem}
\begin{proof}
	Let
	$$
	\Theta_{opt}(R, \mu, \nu):= \{(Q_{opt}(\tilde{\pi}), Q_{opt})\colon\
	\tilde{\pi}\in \tilde{\Pi}(\tilde{\mu}, \tilde{\nu})\},
	$$
	where $Q_{opt}$ is as in Lemma \ref{existence of optimal Markov kernel}.
	By this Lemma, $\Theta_{opt}(R, \mu, \nu)\subseteq \Theta(R, \mu, \nu)$, thus,
	$$
	\inf\left\{\int Q_{\pi}(c) d\tilde{\pi}\colon\  (\pi, Q_{\pi}) \in
	\Theta(R,\mu, \nu)\right\}\leq \inf\left\{\int Q_{opt}(c) d\tilde{\pi}\colon\  \tilde{\pi}\in \tilde{\Pi}(\tilde{\mu},\tilde{\nu})\right\}.
	$$
	As follows from Lemma \ref{existence of a decomposition for a plan},
	$$
	\inf\left\{\int c d\pi\colon\  \pi \in \Pi_R(\mu, \nu)\right\}=\inf\left\{\int Q_{\pi}(c) d\tilde{\pi}\colon\  (\pi, Q_{\pi}) \in \Theta(R,\mu, \nu)\right\}.
	$$
	By Lemma \ref{existence of optimal Markov kernel},
	$\int c dQ_{opt}(\tilde{\pi})\leq \int c dQ_{\pi}(\tilde{\pi})$ for all
	$(\pi, Q_{\pi})\in \Theta(R,\mu, \nu)$. Hence,
	$$
	\inf\left\{\int Q_{\pi}(c) d\tilde{\pi}\colon\  (\pi, Q_{\pi}) \in \Theta(R,\mu, \nu)\right\}\geq \inf\left\{\int Q_{opt}(c) d\tilde{\pi}\colon\  \tilde{\pi}\in \tilde{\Pi}(\tilde{\mu},\tilde{\nu})\right\}.
	$$
	Thus, we conclude
	$$
	\inf\left\{\int c d\pi\colon\  \pi \in \Pi_R(\mu, \nu)\right\}=\inf\left\{\int Q_{opt}(c) d\tilde{\pi}\colon\  \tilde{\pi}\in \tilde{\Pi}(\tilde{\mu},\tilde{\nu})\right\}.
	$$
	The equality
	$$
	\inf\left\{\int c d\pi\colon\  \pi \in \Pi_R(\mu, \nu)\right\}=
	\inf\left\{\int \inf \left\{\int c d\pi\colon\  \pi \in \Pi_R(\xi(x), \eta(y))\right\} d\tilde{\pi}\colon\  \tilde{\pi}\in \tilde{\Pi}(\tilde{\mu},\tilde{\nu})\right\}
	$$
	follows from the explicit form of $Q_{opt}$, described in Lemma \ref{existence of optimal Markov kernel}.
\end{proof}

\section{Decomposition of Wasserstein-like distances}
\sectionmark{Decomposition of Wasserstein-like distances}
In this section we discuss an applications of the decomposition theorem (Theorem \ref{main theorem}) to the theory of Wasserstein-like distances. 

Let $d\colon\  X\times X \to  \R$  be a given distance function, $\mathrm{Dom}_Q\subseteq {\mathcal P}(X)$ be an ergodic decomposable simplex. Recall,
that by Wasserstein-like distance we mean the function $W_p^R\colon\ \mathrm{Dom}_Q\times \mathrm{Dom}_Q \to  [0,\infty]$, defined by the formula (\ref{generalized kantorovich metric}).

For any $[0, +\infty]$-valued distance function $\bar{d}$
on the set $\partial_e\mathrm{Dom}_Q$ of extreme points of $\mathrm{Dom}_Q$
define an extension of this function to the entire $\mathrm{Dom}_Q$:
\begin{equation}
\label{extension of distance function}
\hat{d_p}(\mu,\nu):=\inf\left\{\left(\int \bar{d}^p(\xi, \eta)d\pi\colon\right)^{\frac{1}{p}}\  \pi\in \Pi(\tilde{\mu},\tilde{\nu})\right\},
\end{equation}
where $\tilde{\mu}, \tilde{\nu}\in {\mathcal P}(\partial_e\mathrm{Dom}_Q)$, $\Pi(\tilde{\mu},\tilde{\nu}):=\{\pi\in {\mathcal P}(\partial_e\mathrm{Dom}_Q\times \partial_e\mathrm{Dom}_Q)\colon\  \mathrm{Pr}_1(\pi)=\tilde{\mu}, \mathrm{Pr}_2(\pi)=\tilde{\nu}\}$, $\mathrm{bar}(\tilde{\mu})=\mu$, $\mathrm{bar}(\tilde{\nu})=\nu$.

We are able to formulate the following statement.
\begin{theorem}[Decomposition of a restricted Wasserstein distance]
	\label{theorem restricted wasserstein distance on a simplex}
	Let $(X,d)$ be a metric Polish space equipped with Borel
	$\sigma$-algebra ${\mathcal A}$, $\mathrm{Dom}_Q\subseteq {\mathcal P}(X)$ be a closed ergodic decomposable simplex with the associated Markov kernel $Q$, such that
	$\{Q^x\}=\partial_e\mathrm{Dom}_Q$, ${\mathcal A}^0$ 
	be a corresponding $H$-sufficient $\sigma$-subalgebra, $R=(\Omega, \mathrm{Dom}_Q, \mathrm{Dom}_Q)$ be a \textit{geometric}, \textit{ergodic decomposable}, and \textit{coherent} linear restriction w.r.t. $M$, where $M\subseteq \mathcal P(X\times Y)$ is the corresponding ergodic decomposable simplex, 
	$Q_M$ and $({\mathcal A} \otimes {\mathcal B})^0\subseteq {\mathcal A} \otimes {\mathcal B}$ are the associated Markov kernel and $H$-sufficient $\sigma$-algebra.
	
	Then for $\bar{d}:=W_p^R|_{\partial_e\mathrm{Dom}_Q}$
	(the restriction of $W_p^R$ to the set $\partial_e\mathrm{Dom}_Q$) it is true that
	$$
	W_p^R(\mu,\nu)=\hat{d}_p(\mu,\nu),\ \forall \mu,\nu\in \mathrm{Dom}_Q,
	$$
	where $\hat{d}_p$ is defined by the formula (\ref{extension of distance function}).
	Moreover, $W_p^R$ is an $[0, +\infty]$-valued distance function on $\mathrm{Dom}_Q$.
\end{theorem}
\begin{proof}
	Since $R=(\Omega, \mathrm{Dom}_Q, \mathrm{Dom}_Q)$
	is a geometric restriction, it follows that it is weakly regular.
	It is known, that $d^p$ is bounded below and is a lower semi-continuous function on $X\times X$. Hence $\pi \to  \int d^p d\pi$ is a weakly lower semi-continuous functional on ${\mathcal P}(X\times X)$
	(w.r.t. topology of weak convergence). Since the hypothesis of Theorem \ref{main theorem} is satisfied,
	one can apply it to obtain 
	$$
	\inf\left\{\int d^p d\pi\colon\
	\pi \in \Pi_R(\mu, \nu)\right\}=\inf\left\{\int \inf \left\{\int d^p d\pi\colon\
	\pi \in \Pi_R(\xi(x), \eta(y))\right\} d\tilde{\pi}\colon\
	\tilde{\pi}\in \tilde{\Pi}(\tilde{\mu},\tilde{\nu})\right\}.
	$$
	Here $\tilde{\Pi}(\tilde{\mu},\tilde{\nu})$ is as in Definition \ref{definition tilde pi}. Since
	\begin{equation}
	(x,y)\to \left\{\int d^p d\pi\colon\ \pi \in \Pi_R(\xi(x), \eta(y))\right\}
	\end{equation}
	is measurable w.r.t. ${\mathcal A}^0\otimes {\mathcal B}^0$, 
	and ${\mathcal A}^0\otimes {\mathcal B}^0\subseteq ({\mathcal A}\otimes {\mathcal B})^0$, 
	it is true, that one can replace $\tilde{\Pi}(\tilde{\mu},\tilde{\nu})$ with ${\Pi}(\tilde{\mu},\tilde{\nu})$ 
	without any change of the infimum:
	$$
	\inf\left\{\int d^p d\pi\colon\
	\pi \in \Pi_R(\mu, \nu)\right\}=\inf\left\{\int \inf \left\{\int d^p d\pi\colon\
	\pi \in \Pi_R(\xi(x), \eta(y))\right\} d\tilde{\pi}\colon\
	\tilde{\pi}\in {\Pi}(\tilde{\mu},\tilde{\nu})\right\}.
	$$
	By the definitions of the distances $\hat{d}_p$, $W_p^R$ and due to the established fact about equivalence of $\sigma$-algebras on the space of measures
	(Proposition \ref{relation between weak and evaluation algebras}), the obtained equality is equivalent to the equality
	$$
	W_p^R=\hat{d}_p.
	$$
	It follows from Proposition \ref{geometricity implies distance function} that $W_p^R(\mu,\nu)$ is actually $[0, +\infty]$-valued distance function
	on $\mathrm{Dom}_Q$.
\end{proof}

In can be noted, that the example of decomposition described in the introduction section appears to be a particular case of the just proved statement.

\section*{Acknowledgment}
The paper was written during an internship of the author in Scuola Normale Superiore di Pisa. The author is grateful to SNS for its hospitality, Dr. Luigi Ambrosio for interesting and fruitful mathematical conversations, and Higher School of Economics for a sponsorship of this visit.

The author wish to thank V. I. Bogachev, A. M. Vershik, and A. V. Kolesnikov for important remarks and for the help with preparation of this paper.

\end{document}